\documentclass[11pt,english]{amsart}

\usepackage{amsmath,leftidx,amsthm}
\usepackage[all]{xy}
\usepackage{xspace}
\usepackage[psamsfonts]{amssymb}
\usepackage[latin1]{inputenc}
\usepackage{graphicx,color,mathrsfs}
\usepackage{hyperref,appendix}
\usepackage{xcolor}

\newtheorem{theorem}{\bf Theorem}[section]
\newtheorem{proposition}[theorem]{\bf Proposition}
\newtheorem{definition}[theorem]{\bf Definition}

\newtheorem{lemma}[theorem]{\bf Lemma}
\newtheorem{corollary}[theorem]{\bf Corollary}

\newtheorem{remark}[theorem]{\bf Remark}
\newtheorem{fact}[theorem]{\bf Fact}

\def\C{{\mathbb C}}
\def\N{{\mathbb N}}
\def\R{{\mathbb R}}

\def\Q{{\mathbb Q}}

\def\D{\mathbb{D}}

\def\p{\mathbb{P}}\def\P{\mathbb{P}}

\def\supp{\textup{supp}}

\def\bif{\textup{bif}}

\title{Julia sets and bifurcation loci}
\author{Thomas Gauthier}
\address{Laboratoire de Math\'ematiques d'Orsay, B\^atiment 307, Universit\'e Paris-Saclay, 91405 Orsay Cedex, France}
\email{thomas.gauthier@universite-paris-saclay.fr}
\author{Gabriel Vigny}
\address{LAMFA, Universit\'e de Picardie Jules Verne, 33 rue Saint-Leu, 80039 AMIENS Cedex 1, FRANCE}
\email{gabriel.vigny@u-picardie.fr}

\thanks{The first author's research is partially supported by the Institut Universitaire de France}
\thanks{The second author's research is partially supported by the ANR QuaSiDy, grant [ANR-21-CE40-0016]}

\begin{document}
	\begin{abstract}
We prove that several dynamically defined fractals in $\mathbb{C}$ and $\mathbb{C}^2$ which arise from different type of polynomial dynamical systems can not be the same objects. One of our main results is that the closure of Misiurewicz PCF cubic polynomials (the strong bifurcation locus) cannot be the Julia set of a regular  polynomial endomorphism of $\C^2$. We also show that the Julia set of a Hénon map and a polynomial endomorphism cannot coincide.
\end{abstract}

\maketitle

\medskip

\noindent \textbf{Mathematics~Subject~Classification~(2020): 37F10, 37P30, 37F46}

\section{Introduction}
Fractals are intriguing geometric objects whose loose definition involves irregular shapes and  auto-similarity. They can be produced using complex dynamics where they are everywhere as Julia sets of rational maps or as bifurcation loci of family of rational maps such as the famous boundary $\partial \mathcal{M}$ of the Mandelbrot set. Going deeper, one can wonder whereas two fractals belong to the same family. More precisely, 
a long standing rigidity's problem in complex dynamics was to show that $\partial \mathcal{M}$ is never the Julia set of a rational map. This was recently shown by Ghioca, Krieger and Nguyen for polynomial maps \cite{Ghioca_Krieger_Nguyen}. To do that, they show that the harmonic (or bifurcation) measure of the Mandelbrot is not the equilibrium measure of a polynomial. Luo \cite{Luo} extended their result by showing that $\partial \mathcal{M}$ is not the Julia set of a rational map. He shows it is a consequence of the following fact: while Julia sets have many local symmetries, $\partial \mathcal{M}$ does not. This relies deeply on Lei's similarity between $\partial \mathcal{M}$  and the Julia set of some polynomials  \cite{Tan-similarity}.   

\begin{figure}[htbp]
	\includegraphics[height=4cm]{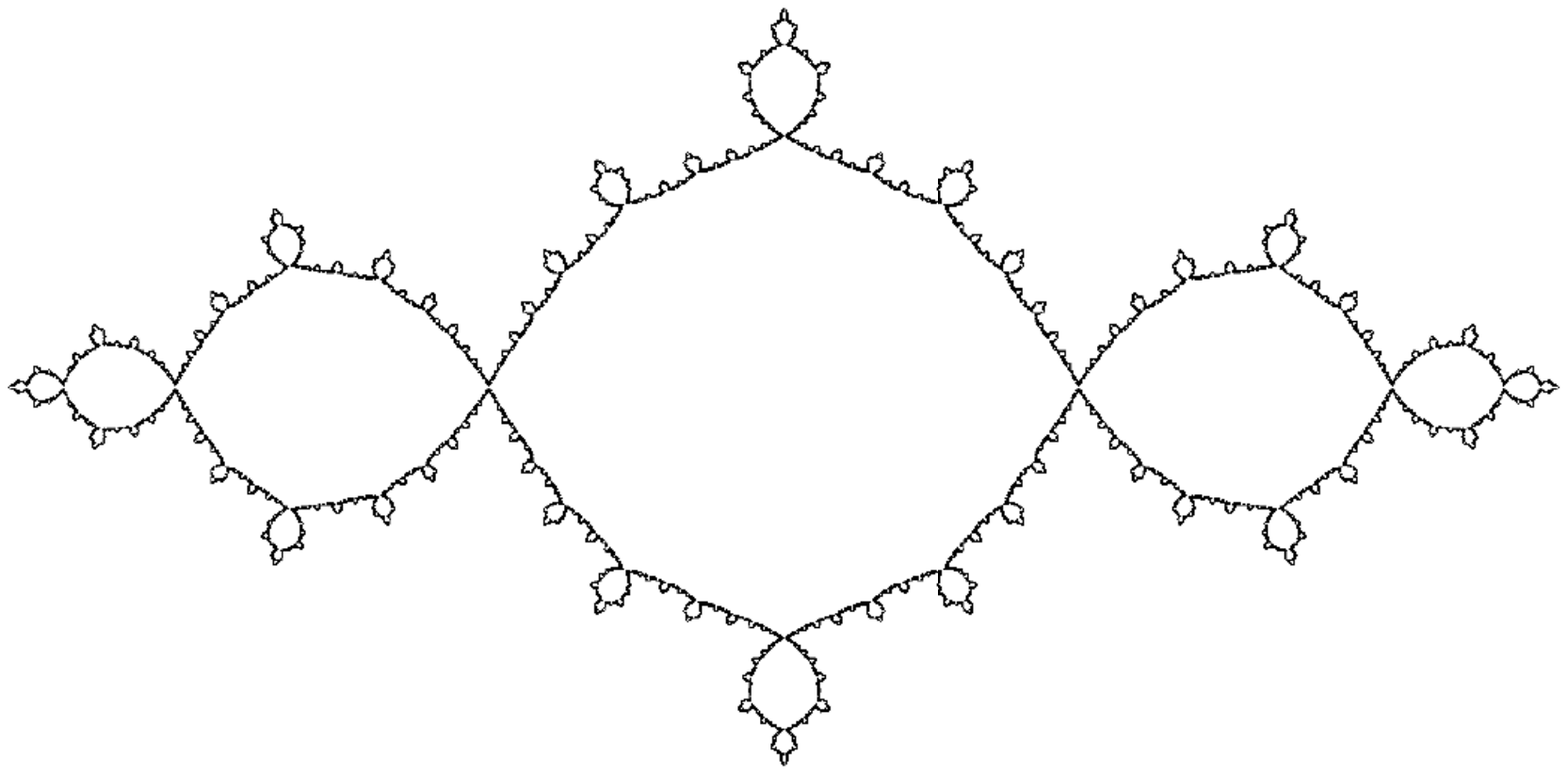}
	\includegraphics[height=4cm]{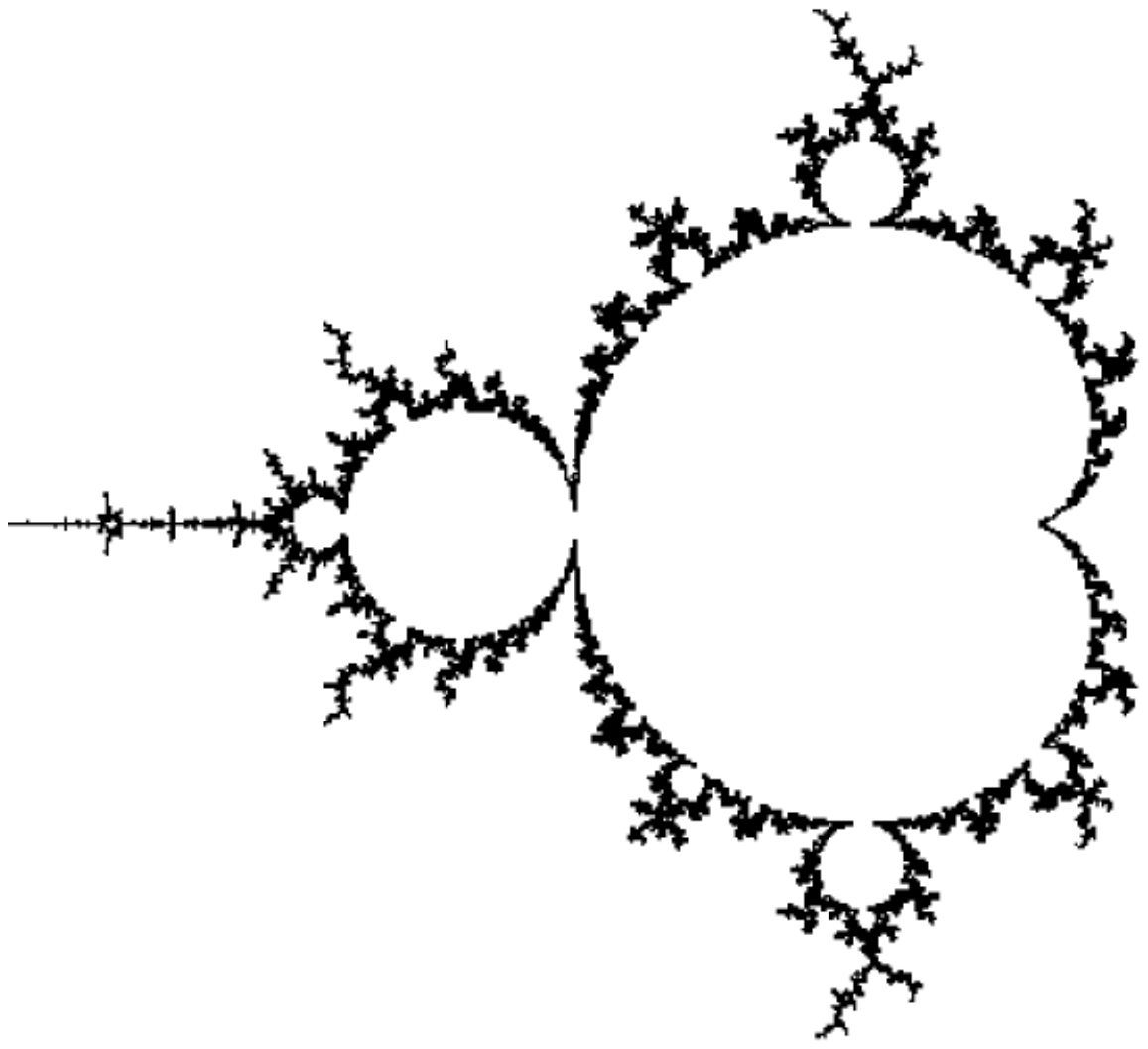}
	\caption{Left: the Julia set of $z^2-1$. Right: $\partial \mathcal{M}$.}
\end{figure}

Many similar rigidity's results were a driving force in complex dynamics, with striking applications in arithmetic dynamics. As in \cite{Ghioca_Krieger_Nguyen} above, fractals in complex dynamics carry canonical invariant measures and one can express rigidity statements in term of sets (when do different objects produce the same fractal?) or in term of measures  (when do different objects produce the same measure?). Notably  
\begin{itemize}
	\item Beardon~\cite{Beardon-sym,Beardon-sameJ} showed that polynomials sharing their Julia set have to share an iterate, up to affine symmetry of the Julia set (see also \cite{schmidt-steinmetz}). Levin-Przyticki~\cite{Levin-Przytycki} generalized this result to rational maps, expressing the statement in terms of the Green measures when the Julia set is $\mathbb{P}^1(\mathbb{C})$ (see also~\cite{ye} and \cite{pakovich}).
	\item Building on those results, Baker-DeMarco~\cite{BDM1} showed rational maps having an infinite set of preperiodic points in common must have the same Julia set and Green measure. DeMarco-Mavraki~\cite{DeMM} recently made this property uniform: for any two generic rational maps, there is a  uniformly bounded number of preperiodic points in common.
	\item Looking at local rigidity properties, Levin~\cite{Levin-sym} showed that a Julia set can not have too many local symmetries. Building on this result, Dujardin-Favre-Gauthier~\cite{DFG} showed that under mild assumptions, a local symmetry of a Julia set comes from a global symmetry, and Ji-Xie~\cite{Ji-Xie-local} generalized this result.
	\end{itemize}

The purpose of the article is to address similar questions in higher dimension, especially in dimension $2$. Note that, while rational maps of the Riemann sphere of degree $D$ are always holomorphic of topological degree $D$, the situation is wilder in $\C^2$ with the possibility of indeterminacy sets and various relations for the algebraic and topological degrees \cite{Sibony}. Furthermore, the bifurcation locus can be stratified in several parts, in terms of supports of positive closed currents. We chose here to consider only the part with maximal bifurcation (so we stay in the formalism of measures) and we consider only three paradigmatic cases. The following definitions we give in the introduction for the different fractal sets are not the usual ones, though they are equivalent, but it allows us to state our results without resorting to pluripotential theory.   
\begin{enumerate}
	\item $h:\C^2 \to \C^2$, a polynomial endomorphism of $\P^2(\C)$ of degree $D > 1$: it is a polynomial map of $\C^2$ whose homogeneous extension to $\P^2$ is holomorphic. The (small) Julia set $J_h$ we consider is contained in the accumulation set of $\mathrm{Preper}(h):=\{z\in \C^2, \ \exists \, n,k \in \N^*\times \N, \ h^{n+k}(z)=h^k(z) \}$ which is the set of points with finite orbit for $h$. More precisely, if $R_n$ denote the set of periodic points of period $n$ ($h^n(z)=z$) which are repelling (the eigenvalues of $D(f^n)(z)$ have modulus $> 1$), then the sequence of measures
	\[\lim_n D^{-2n} \sum_{z\in R_n} \delta_z\] 
	converges in the sense of measure to a measure $\mu_h$ and $J_h= \supp(\mu_h)$ (\cite{briendduval2}).
	\item for $x,y\in \C^2$, we consider the polynomial 
	\[ P_{x,y}(z)= \frac{1}{3}z^3-\frac{x}{2}z+y \]    
	so that $(x,y) \mapsto P_{x,y}$ is an orbifold parametrization of the cubic polynomials. The \emph{marked critical points} are $c_1(x,y)= 0$ and $c_2(x,y)=x$. A parameter $(x,y)$ is post-critically finite, PCF for short, if there exist $(n_1,m_1), (n_2,m_2)\in \N \times \N^*$ such that $P_{x,y}^{n_{i}+m_{i}}(c_i(x,y))= P_{x,y}^{n_{i}}(c_i(x,y))$ for $i=1,2$. Let $PCF\subset \C^2$ denote the set of PCF cubic maps. The strong bifurcation locus $S_\bif$ is defined here as the accumulation set of $\mathrm{PCF}$ \cite{buffepstein}.
	\item $f:\C^2\to \C^2$, a (generalized) H\'enon map: it is a polynomial automorphism of $\C^2$ of degree $d>1$ such that its homogeneous extension (resp. the extension of $f^{-1}$) to $\P^2$ admits an indeterminacy point $I^+$ (resp. $I^-$) with $I^+\neq I^-$. The Julia set $J_f$ of $f$ is defined as the accumulation set of $\mathrm{Per}(f):=\{z\in \C^2, \ \exists \, n\in \N^*, \ f^{n}(z)=z \}$ \cite{BedfordLyubichSmillie}.
\end{enumerate}  
Our main result is 
\begin{theorem}\label{Support} Let $h:\C^2 \to \C^2$ be a polynomial endomorphism of $\P^2(\C)$ of degree $D > 1$, $f:\C^2\to \C^2$, a  H\'enon map of degree $d>1$. Then:
	\[   J_h \neq S_\bif \neq J_f \neq J_h.\]
\end{theorem}
Similarly to \cite{Ghioca_Krieger_Nguyen} where the result is proved by looking at equilibrium measures, Theorem~\ref{Support} is a consequence of Theorem~\ref{tm_Functions} below where we look at equilibrium measures defined in term of pluripotential theory. One common idea behind the proofs is to promote the fact that two sets are equal (e.g. $J_h=J_f$) to the fact that two measures are equal and then to a similar property on some positive closed currents. 

 For $J_h\neq J_f$, we use the laminar structure of the Green current of $f$ and a rigidity result of Cantat and Xie \cite{cantat2020birational} on birational conjugacy of polynomial endomorphisms. The case $J_f\neq S_\bif$ is the easiest and follows from the fact that the Green current of a Hénon map is extremal due to Forn\ae ss and Sibony \cite{FS2} whereas the bifurcation current of a marked point is not.

The most challenging case is $S_\bif \neq J_h$ where we follow Luo's approach and show a rigidity result on families of polynomial maps with a marked point having biholomorphic bifurcation measures (see Proposition~\ref{lm:at-transversely-prerepelling} below).
We study for that local symmetries of bifurcation loci of special one-dimensional families of polynomials. A family of polynomials parametrized by a quasi-projective curve $\Lambda$ is a polynomial map $P:(t,z)\in\Lambda\times \Lambda\mapsto P_t(z)\in\C$ such that $P_t$ is a degree $d\geq2$ polynomial for all $t\in \Lambda$. Such a family is special if there are infinitely many $t\in \Lambda$ such that $P_t$ is PCF, i.e. such that $\bigcup_n P_t^n(\mathrm{Crit}(P_t))$ is an infinite set. In such a family, the bifurcation locus is the accumulation set of the set $\{t\in \C\, ; \ P_t \ \text{is PCF}\}$. 
\begin{theorem}\label{tm:special}
	The bifurcation locus of a special family parametrized by $\C$ is not the Julia set of a polynomial map. 
\end{theorem}
A famous example of special families parametrized by $\C$ is the unicritical family (which is covered by Ghioca, Krieger and Nguyen and by Luo) but there are other relevant examples such as $(t,z)\mapsto z^3-tz^2$ (see,~e.g.,~\cite{MilnorCubic}).

While the article is certainly the first to establish rigidity properties of Julia sets and bifurcation locus coming from different algebraic worlds in dimension $>1$ (Hénon, polynomial endomorphisms and family of cubic polynomials), there have been significant works on the question of understanding maps coming from a same world that share a strong relationship 
\begin{itemize} 
	\item Dinh classified pairs of commuting polynomial endomorphisms of $\P^2(\C)$ \cite{Dinh_permutable}, showing that they either share an iterate or satisfies a specific algebraic equation (see also Dinh and Sibony \cite{Dinh_Sibony_permutable},
 and  Kaufmann \cite{Kaufmann_permutable} for the general case of  endomorphisms of $\P^2(\C)$). 
	\item Generalizing Baker and DeMarco, Yuan and Zhang \cite{Yuan_Zhang_arithmetic_Hodge} showed that polarized endomorphisms defined over a number field share a Zariski dense set of preperiodic points if and only if they have the same Green measure at every place (in particular, at the complex place). 
	\item For complex Hénon maps, Dujardin and Favre showed, in an article which inspired us deeply \cite{DujardinFavre_MM},  that two Hénon maps have the same Julia set if and only if they share an iterate. They rely for that on a rigidity result of Lamy \cite{Lamy_henon} that states two Hénon maps which have the same forward Julia set (set of points of bounded forward orbit) share an iterate.     
	\item Families having similar bifurcation measure in higher dimension has not been investigated and is probably a challenging topic. It should have strong links with the ``Ultimate Dynamical Manin Mumford'' proposed by DeMarco and Mavraki. 
\end{itemize}

As a consequence of the theorem of equidistribution of points of small height (\cite{Lee-Henon} for H\'enon maps and \cite{yuan} for endomorphisms), and of the solution to the Manin-Mumford problem for plane automorphism of Dujardin and Favre \cite{DujardinFavre_MM},  we deduce:
\begin{corollary}\label{cor_unif_henon_polynomial}
	Let $h:\C^2 \to \C^2$  a polynomial endomorphism of $\P^2(\C)$ of degree $D>1$ and let $f:\C^2\to \C^2$ be a generalized H\'enon map of degree $d>1$ with $|\mathrm{Jac}(f)| \neq 1$ then 
	\begin{enumerate}
		\item  $ \mathrm{Preper}(h)\cap \mathrm{Per}(f)$ is a finite set
		\item  $ \mathrm{Per}(f)\cap PCF$ is a finite set.
	\end{enumerate}
\end{corollary}
 
Relying on the recent work~\cite{DFR} of Dujardin, Favre and Ruggiero on the Manin-Mumford Conjecture for polynomial endomorphisms, we also get the following corollary.
\begin{corollary}\label{cor_unif_bif_polynomial}
Let $h:\C^2 \to \C^2$  a polynomial endomorphism of $\P^2(\C)$ of degree $D>1$. Assume $h$ is defined over a number field. Assume also the restriction $h|_{L_\infty}$ of $h$ to the line $L_\infty$ at infinity of $\mathbb{C}^2$ in $\mathbb{P}^2(\mathbb{C})$ has no super-attracting periodic points. 
	
	Then $\mathrm{Preper}(h)\cap  PCF$ is a finite set.

\end{corollary}

 In section~\ref{the cast}, we will define the different players, notably using pluripotential theory and explain how to reduce the proof of Theorem~\ref{Support} to showing that some (psh) functions cannot coincide, see Theorem~\ref{tm_Functions} below. In section~\ref{sec:endo-vs-bif}, we prove that the strong bifurcation locus of cubic polynomials can not be the small Julia set of a regular polynomial endomorphism of $\mathbb{C}^2$ and that the bifurcation locus of a special family can not be the Julia set of a rational map. In section~\ref{sec:Henon}, we show that the Julia set of a Hénon map can not be the Julia set of a polynomial endomorphism or the strong bifurcation locus of cubic polynomials.

\section{The cast}\label{the cast}
We define the cast and use pluripotential theory only starting from here.

Let us recall some basic fact on pluripotential theory \cite{Demailly}.
Let $\Omega\subset \C^2 $ be an open subset, we say that a function $\varphi : \Omega \to \R\cup\{-\infty\}$ is plurisubharmonic (psh for short) if its restriction to any complex line is subharmonic (a subharmonic function is the limit of a decreasing sequence of $\mathscr{C}^2$ functions with non negative Laplacian). Then, if $\varphi$ is not identically $-\infty$, it is in $L^1_\mathrm{loc}$ so that we can defined $dd^c \varphi$ in the sense of currents. It is a positive closed current of bidegree $(1,1)$. Now for a positive closed current $T$ of bidegree $(1,1)$, if $\varphi$ is locally bounded, we can defined $dd^c \varphi \wedge T := dd^c (\varphi \wedge T)$ in the sense of currents and this is a positive measure.  

When $\Omega =\C^2$, we say that a function $\varphi$ is in \emph{the Lelong class} $\mathcal{L}(\C^2)$ if it is psh and $\varphi \leq \log^+\|z\| +O(1)$ where $\log^+\|z\|= \max (\log\|z\|,0)$, we say that it is in $\mathcal{L}^+(\C^2)$ if in addition $\varphi=\log^+\|z\| +O(1)$. 
Assume that $\varphi\in \mathcal{L}^+(\C^2)$ so that $\varphi$ is locally bounded and we can define $(dd^c \varphi)^2$ which is then a probability measure. We shall be using the following classical fact:
\begin{fact}\label{fact2} \normalfont
Let $K\subset \C^2$ be a compact set such that there exists a psh function $\varphi\in \mathcal{L}^+(\C^2)$ with $ \supp (dd^c \varphi)^2\subset K$ and $\varphi=0$ on $K$, then such $\varphi$ is unique. More precisely, for any $\psi \in \mathcal{L}(\C^2)$ with $\psi \leq 0$ on $K$, then $\psi \leq \varphi$ on $\C^2$ (see e.g. \cite[Proposition 6.14]{favredujardin}). 
\end{fact}

\subsection{Polynomial endomorphisms}\label{sec_endo}
Let $h:\C^2 \to \C^2$ be a polynomial endomorphism of $\P^2(\C)$ of degree $D > 1$: it is a polynomial map of $\C^2$ whose homogeneous extension to $\P^2$ is holomorphic. For such maps, Forn\ae ss and Sibony \cite{FS2} defined and studied the equilibrium measure $\mu_h$ of $h$ where $\mu_h:= (dd^c G_h)^2$ where $G_h(z):= \lim_n D^{-n}\log^+\|h^n(z)\|$ is the (psh) Green function of $f$, $G_h \in \mathcal{L}^+(\C^2)$, and $G_h$ is H\"older continuous (\cite{Sibony}).  It is an invariant and mixing probability measure of entropy $2\log D$. As explained in the introduction, Briend and Duval showed in \cite{briendduval} that it equidistributes the $J_h$-repelling periodic points in the following sense: let $R_n$ denote the set of periodic points of period $n$ which are repelling, then 
\[\mu_h = \lim_n D^{-2n} \sum_{z\in R_n} \delta_z,\]
so $J_h= \supp(\mu_h)$. Furthermore, we have $G_h=0$ on $\supp(\mu_h)$ so by Fact~\ref{fact2}, $G_h$ is the unique psh function in the Lelong class such that $\supp(dd^c G_h)^2 \subset J_h$ and $\varphi=0$ on $J_h$. 
 
We also denote $T_h$ the \emph{Green current} of $h$ which is defined by $\lim_n D^{-n}(h^n)^*(\omega)$, where $\omega$ is the Fubini-Study form on $\P^2$. It is a positive closed current of bidegree $(1,1)$
 and it is the trivial extension of $dd^c G_h$ to $\P^2$ (\cite{Sibony}). 

The following lemma will help us improve a local invariance of a Green function to a global one.
\begin{lemma}\label{lm:preserving} Let $g\geq 0$ be a psh function in $\C^2$ such that $g$ is pluriharmonic on the connected open set $\Omega:= \{g>0\}$. Let $U\subset \Omega$ be a connected non empty open set. Assume  $h^{-1}(\Omega)=\Omega$ and $h^{-1}(U)=U$. Assume also that $g\circ h =D h$ on $U$ then $g\circ h = Dh 
	$ on $\Omega$. 
\end{lemma}
\begin{proof}
	For $a>0$, we consider $g_a:=\max(g ,a)$ and $T_a:= dd^c g_a$, then $T_a$ is a positive closed current on $\Omega$, supported on $\{g=a\}$, and $T_a^2= dd^c (g_a \wedge T_a)= 0$. Furthermore, on any ball $B \subset \Omega$, we can write $g= \ln |\varphi_B|$ where $\varphi_B$ is a non vanishing holomorphic function. In particular, on $B$   
	\[ T_a = \int_{0}^{2\pi} [\varphi_B = e^a e^{ i \theta} ] \frac{d\theta}{2\pi},\] 
	where for each $\theta$, $[\varphi_B = e^a e^{ i \theta} ]$ denotes the current of integration on the analytic set $\{\varphi_B = e^a e^{ i \theta} \}$. The analytic sets $\{\varphi_B = e^a e^{ i \theta}\}_{\theta, B}$ are thus local leaves of a global lamination of $\{g=a\}$ by (immersed) Riemann surfaces, possibly singular.   
	
	On $U$, as $g \circ h =D g$ we have that $g_{Da} \circ h =D g_{a}$ so $h^* T_{Da}= D T_a$.
	Take in particular $B\subset U$ such that $h(B)$ is contained in a ball $B'$ (up to reducing $B$). On $B$, we have $h^* T_{Da} \wedge T_a = 0$, so by \cite{Dujardin_laminar}, we deduce that $h$ sends each leaf $\{\varphi_B(y)=\exp(a +i\theta)\}$ to a leaf $\{\varphi_{B'}(y)=\exp(Da +i\theta') \}$. 
	
	Summing up, we have that, in $U$, the lamination of $\{g=a\}$ given by the maps $\varphi_B$ is sent by $h$ to the the lamination of $\{g=Da\}$. By analyticity, the global lamination of $\{g=a\}$ is sent by $h$ to the global lamination of $\{g=Da\}$. So on the whole $\Omega$, we have $g \circ h= Dg$. 
\end{proof}
\subsection{Hénon maps}\label{sec_henon}
Let $f:\C^2\to \C^2$ be a (generalized) H\'enon map: it is a polynomial automorphism of $\C^2$ of degree $d>1$ such that its homogeneous extension (resp. the extension of $f^{-1}$) to $\P^2$ admits an indeterminacy point $I^+$ (resp. $I^-$) with $I^+\neq I^-$. The ergodic study of $f$ was developed by Bedford and Smillie, with Lyubich, in a seminal series of papers \cite{BedfordSmillie1, BedfordSmillie3, BedfordLyubichSmillie}. They notably defined the \emph{equilibrium  measure} $\mu_f$ of $f$ whose construction and properties we briefly recall:
\begin{itemize}
	\item the plurisubharmonic (psh for short) functions $G_f^+(z):= \lim d^{-n} \log^+ \| f^n(z)\|$ and $G_f^-(z):= \lim d^{-n} \log^+ \| f^{-n}(z)\|$ are well defined (where $\log^+ x =\max (\log x, 0)$). In particular, the \emph{Green currents} $T_f^+:= dd^c G_f^+$ and  $T_f^-:= dd^c G_f^-$, taking the $dd^c$ on $\C^2$, are well defined positive closed currents of mass $1$.
	\item The intersection of the supports $J_f^+$ and $J_f^-$ of $T_f^+$ and $T_f^-$ is a compact set in $\C^2$ hence $\mu_f:=T_f^+\wedge T_f^-$ is a well defined probability measure. Alternatively, we can also write $\mu_f:= (dd^c G_f)^2$ where $G_f :=\max(G_f^+, G_f^-)\in  \mathcal{L}^+(\C^2)$ is  H\"older continuous. 
	\item The measure $\mu_f$ is invariant, mixing and hyperbolic (non zero Lyapunov exponent). It has maximal entropy $\log d$ and it equidistributes the saddle periodic points in the following sense: let $S_n$ denote the set of periodic points of period $n$ ($f^n(z)=z$) which are saddle (the eigenvalues of $D(f^n)(z)$ have modulus $\neq 1$), then 
	\[\mu_f = \lim_n d^{-n} \sum_{z\in S_n} \delta_z.\]  
\end{itemize} 
In particular, the set $J_f$ defined in the introduction is equal to $\supp(\mu_f)$ and from $J^\pm_f \subset\{G_f^\pm=0\}$, one deduces $G_f=0$ on $J_f$. Hence, the function $G_f$ is the unique psh function in $\mathcal{L}^+(\C^2)$ such that $\supp (dd^cG_f)^2\subset J_f$ and $G_f=0$ on $J_f$ by Fact~\ref{fact2}.

\subsection{Family of degree \textit{d} polynomials}\label{sec_cubic}
Following Dujardin and Favre \cite{favredujardin}, for $d\geq3$, we can define a good family of degree $d$ polynomials parametrized by $\C^{d-1}$ as follows: for $(x_1,\ldots,x_{d-2},y)\in \C^2$, we consider the polynomial 
\[P_{x,y}(z):=\frac{1}{d}z^d+\sum_{j=2}^{d-1}(-1)^{d-j}\sigma_{d-j}(x)\frac{z^j}{j}+y~, \ z\in\C~,\]
where $\sigma_{\ell}(x)$ is the monic symmetric polynomial of degree $\ell$ in $(x_1,\ldots,x_{d-2})$, so that $(x,y) \mapsto P_{x,y}$ is an orbifold parametrization of the moduli space of degree $d$ polynomials. The \emph{marked critical points} are $c_1(x,y)= 0$ and $c_i(x,y)=x_{i-1}$ for $2\leq i\leq d-1$. We can consider on $\C^{d-1}$ the  \emph{activity current} $T_i$ of $c_i$ for $i=1,2$ which is defined by 
$T_i:=dd^c G_i$ (the $dd^c$ is in the $(x,y)$ variables) where
\[G_i:= \lim_n \frac{1}{d^n} \log^+|P_{x,y}^n(c_i)|.\]   
Then $T_1$ (resp. $T_i$, $2\leq i\leq d-1$) is a positive closed current of mass $1/d$ (resp. $1$) in $\P^{d-1}$ and the psh function
\[G_\bif:=\max(dG_1,G_2,\ldots,G_{d-1})\]
is in $\mathcal{L}^+(\C^{d-1})$ by \cite{favredujardin} with the same definition for $\mathcal{L}^+(\C^{d-1})$ than $\mathcal{L}^+(\C^{2})$ (one also can prove it induces a continuous metrization on $\mathcal{O}_{\mathbb{P}^{d-1}}(1)$, see \cite{favredujardin,favregauthier}). The \emph{bifurcation measure} $\mu_\bif$ is defined by 
\[\mu_\bif:=  (dd^c G_\bif)^{d-1}=d\cdot \bigwedge_{j=1}^{d-1}dd^cG_j. \]
One also has $G_\bif=0$ on $\supp(\mu_\bif)$. In particular, when $d=3$, with the notations of the introduction, we have $S_\bif= \supp(\mu_\bif)$  and $G_\bif$ is the unique psh function in the Lelong class such that $(dd^c G_\bif)^2 \subset S_\bif$ and $G_\bif=0$ on $S_\bif$.

\medskip

A parameter $(x,y)$ is post-critically finite, PCF for short, if for any $1\leq i\leq d-1$, there exist $(n_i,m_i)\in \N \times \N^*$ such that $P_{x,y}^{n_{i}+m_{i}}(c_i(x,y))= P_{x,y}^{n_{i}}(c_i(x,y))$ for all $i$. Let $PCF\subset \C^{d-1}$ denote the set of PCF degree $d$ polynomials.

By the result of Favre and the first author \cite{favregauthier} (see also \cite{Good-height, YZ-adelic}), we have that if one takes a Zariski dense sequence of PCF parameters $(x_k,y_k)_k$, then one can extract a subsequence such that ($\mathrm{Orb}(x_k,y_k)$ is the Galois orbit of $(x_k,y_k)$)
\[\lim_k \frac{1}{ \# \mathrm{Orb}(x_k,y_k)}  \sum_{(x,y) \in \mathrm{Orb}(x_k,y_k) } \delta_{(x,y) } \to \mu_\bif. \] 

\subsection{Restating Theorem~\ref{Support} in term of measures and potentials}
From the uniqueness of the functions $G_h$, $G_\bif$ and $G_f$, we see that the following theorem is in fact equivalent to Theorem~\ref{Support}. 
\begin{theorem}\label{tm_Functions} Let $h:\C^2 \to \C^2$ be a polynomial endomorphism of $\P^2(\C)$ of degree $D > 1$, $f:\C^2\to \C^2$, a  H\'enon map of degree $d>1$. Then:
	\begin{align}
	\mu_h &\neq \mu_\bif \ \mathrm{or \ equivalently} \ G_h\neq G_\bif \label{eq_endo_bif} ;		\\
		\mu_h &\neq \mu_f \ \mathrm{or \ equivalently} \ G_h\neq G_f \label{eq_endo_henon}; 		\\
			\mu_f &\neq \mu_\bif \ \mathrm{or \ equivalently} \ G_f\neq G_\bif \label{eq_henon_bif}. 	
	\end{align}

\end{theorem}

\section{The strong bifurcation locus is not a Julia set} \label{sec:endo-vs-bif}

In this section, we prove Theorem~\ref{tm:special} and the first point of Theorem~\ref{tm_Functions}.
 To do so, we focus first on local biholomorphisms between bifurcation loci in one-dimensional families of polynomial maps. Note that in \cite{Luo}, PCF parameters are necessarily sent to PCF parameters by such a local biholomorphism since PCF parameters are exactly  branching points of the boundary $\partial \mathcal{M}$ of the Mandelbrot set (which is a topological characterization). This has no reason to be true in general so we work with parameters satisfying some strong Collet-Eckmann conditions for which results of Ji and Xie \cite{JiXie} and Dujardin, Favre and the first author \cite{DFG} help us to promote a local symmetry between the bifurcation measures to the fact that pairs of polynomials are intertwined. We then conclude by the work of Favre and the first author ~\cite{book-unlikely} which states that only finitely pairs are intertwined when we fix the first polynomial.    
 To go to the family of cubic polynomial, we use the fact due to Dujardin and Favre \cite{favredujardin} that a stable one dimensional family of cubic maps is special.

\subsection{Bifurcation versus polynomial: local measurable similarity}

Let $P:\Lambda\times\p^1\to\Lambda\times\p^1$ be an analytic family of polynomials of degree $d\geq2$ parametrized by a holomorphic curve $\Lambda$, i.e. $P$ is holomorphic, and for any $(t,z)\in\Lambda\times \mathbb{P}^1$, we have $P(t,z)=(t,P_t(z))$ and the map $P_t$ is a polynomial of degree $d$ for all $t\in \Lambda$. Let
\[G_{P_t}(z):=\lim_{n\to\infty}\frac{1}{d^n}\log^+|P_t^n(z)|, \quad z\in \mathbb{C},\]
be the Green function of $P_t$, and the convergence is local uniform.

\medskip

A polynomial holomorphic pair $(P,a)$ of degree $d$ parametrized by $\Lambda$ is the data of a family $P$ of degree $d$ as above parametrized by $\Lambda$ and of a holomorphic map $a:\Lambda\to\mathbb{C}$.
The bifurcation measure $\mu_{P,a}$ of the pair $(P,a)$ is defined by
\[\mu_{f,a}=dd^c_tG_{P_t}(a(t)).\]

\medskip
In the following, we say:
\begin{enumerate}
\item The family $P$ \emph{satisfies condition} CE with exponent $\lambda>1$ at $t_0\in \Lambda$ if there are a constant $C>0$ and an integer $N\geq1$ such that
\[|(P_{t_0}^n)'(P_{t_0}^N(c))|\geq C\cdot \lambda^n,\]
 for any integer $n\geq0$ and any $c\in \mathrm{Crit}(P_{t_0})\cap J(P_{t_0})$ and if any $c\in \mathrm{Crit}(P_{t_0})\setminus J(P_{t_0})$ lies in an attracting basin.
\item The pair $(P,a)$ \emph{satisfies condition} ParCE with exponent $\lambda>1$ at $t_0\in \Lambda$ if there is a constant $C>0$ such that for any integer $n\geq0$,
\[\left|\left.\frac{\partial P_t^n(a(t))}{\partial t}\right|_{t=t_0}\right|\geq C\cdot \lambda^n.\]
\item The pair $(P,a)$ \emph{satisfies condition} PR$(s)$ with exponent $s>1/2$ at $t_0\in \Lambda$ if there is an integer $N\geq1$ such that for any integer $n\geq N$ and any $c\in \mathrm{Crit}(P_{t_0})$,
\[\left|P_{t_0}^n(a(t_0))-c\right|\geq n^{-s}.\]
\end{enumerate}

Given two positive measures $\mu_1$ and $\mu_2$, we write $\mu_1\asymp\mu_2$
if  $c^{-1} \mu_2\leq \mu_1\leq c\mu_2$ for some positive constant 
$c$. Following the terminology of \cite{book-unlikely},   we say that a polynomial $P$ is \emph{integrable} if it is either a Chebychev or a monomial. In this paragraph, we prove the next Proposition.

\begin{proposition}\label{lm:at-transversely-prerepelling}
Let $(P,a)$ and $(Q,b)$ be polynomial holomorphic pairs of degree $d\geq2$, parametrized by $\Lambda_1$ and $\Lambda_2$ respectively. Assume there exists a biholomorphism $\phi:\Lambda_1\to\Lambda_2$ such that $\mu_{P,a} \asymp  \phi^*(\mu_{Q,b})$. Assume in addition that there is $t_0\in \Lambda_1$ such that $P$ and $Q$ satisfy condition \emph{CE} and $(P,a)$ and $(Q,b)$ satisfy condition \emph{ParCE} and \emph{PR}$(s)$ for some $s>1/2$ at $t_0$ and $\phi(t_0)$ respectively and are not special.

Then there exists an irreducible algebraic curve $Z\subset\mathbb{P}^1\times\mathbb{P}^1$ which projects surjectively onto both coordinates such that $Z$ is preperiodic under iteration of $(P_{t_0},Q_{\phi(t_0)})$.
\end{proposition}

We also say the pair $(P,a)$ is \emph{transversely prerepelling} at $t_0\in\Lambda$ if there are $k\geq0$ and $p\geq1$ such that $x_0:=P_{t_0}^k(a(t_0))$ is $p$-periodic and repelling and if the graph $\Gamma_1$ of $t\mapsto P_t^k(a(t))$ and the graph $\Gamma_2$ of the analytic continuation of $x_0$ as a repelling periodic point of $f_t$ are smooth and transverse local submanifolds of $\Lambda\times\p^{1}$ at $(t_0,x_0)$.

\medskip

The following is an immediate consequence of Proposition~\ref{lm:at-transversely-prerepelling}.


\begin{corollary}\label{cor:at-transversely-prerepelling}
Let $(P,a)$ and $(Q,b)$ be polynomial holomorphic pairs of degree $d\geq2$, parametrized by $\Lambda_1$ and $\Lambda_2$ respectively. Assume there exists a biholomorphism $\phi:\Lambda_1\to\Lambda_2$ and $\alpha>0$ such that $\mu_{f,a}\asymp \phi^*(\mu_{f,b})$. Assume in addition that there is $t_0\in \Lambda_1$ such that $P_{t_0}$ and $Q_{\phi(t_0)}$ are postcritically finite and $(P,a)$ and $(Q,b)$ are transversely prerepelling at $t_0$ and $\phi(t_0)$ respectively.

Then there exists an irreducible algebraic curve $Z\subset\mathbb{C}^2$ which projects surjectively onto both coordinates such that $Z$ is preperiodic under iteration of $(P_{t_0},Q_{\phi(t_0)})$.
\end{corollary}

To prove Proposition~\ref{lm:at-transversely-prerepelling}, we rely on the following asymptotic similarity result stated in \cite[Theorem~5.2]{JiXie} which is an asymptotic similarity property between the parameter space at a parameter $t$ and the phase space of $P_t$ in the spirit of \cite{Tan-similarity}:

\begin{theorem}[Ji-Xie]\label{tm:JiXie}
Let $(P,a)$ be a polynomial holomorphic pair parametrized by $\D$ which satisfies assumptions $(1)$, $(2)$ and $(3)$ above at $0\in \D$. There are a set $A\subset \mathbb{N}$ with density at least $9/10$, a sequence $0<\rho_n<1$, such that $\rho_n\to0$, when $n\in A$ tends to infinity such that, if we let
\[h_n:t\in\mathbb{D}\longmapsto P_{\rho_n t}^{n}(a(\rho_nt))\]
the family $\{h_n:\mathbb{D}\to\mathbb{C},n\in A\}$ is normal and all its limits are non-constant holomorphic maps from $h:\mathbb{D}\to\mathbb{C}$ with $h(0)\in J_{P_0}$.
\end{theorem}

We also rely on the following result~\cite[Theorem~A]{DFG}. 
\begin{theorem}[Dujardin-Favre-Gauthier]\label{tm:DFG}
Let $P$ and $Q$ be degree $d$ polynomials which satisfy condition $\mathrm{CE}$ and are not special. Assume there are $U\subset\mathbb{C}$ with $U\cap J_Q\neq\varnothing$ and $\sigma:U\to \mathbb{C}$ such that $\sigma^*(\mu_Q)\asymp \mu_P$ as measures on $U$. Then there exists an irreducible algebraic curve $Z\subset\mathbb{C}^2$ such that $Z$ is preperiodic under $(P,Q)$ and such that the graph of $\sigma$ is a local branch of $Z$. In particular, $Z$ projects surjectively onto both coordinates.
\end{theorem}

\begin{proof}[Proof of Proposition~\ref{lm:at-transversely-prerepelling}]
Take local charts in $\Lambda_1$ and $\Lambda_2$ such that in these charts, $t_0=0$ and $\phi(t_0)=0$ and $\phi$ induces a biholomorphism $\psi:\D\to\D$ with $\psi(0)=0$.
By Theorem~\ref{tm:JiXie}, there are sets $A_a,A_b\subset \mathbb{N}$ with densities at least $9/10$ and sequences $(\rho_{n,a})_{n\in A_a}$ and $(\rho_{n,b})_{n\in A_b}$ such that the associated families $\{h_{n,a}\}_{n\in A_a}$ and $\{h_{n,b}\}_{n\in A_b}$ are normal families with non-constant limits.
Let $A:=A_a\cap A_b$. Then $A$ has density at least $8/10>0$. The families $\{h_{n,a}\}_{n\in A}$ and $\{h_{n,b}\}_{n\in A}$ are normal families with non-constant limits. We thus can choose $n_j\to\infty$ such that $h_{n_j,a}\to h_a$ and $h_{n_j,b}\to h_b$ uniformly locally on $\mathbb{D}$ and $h_a,h_b:\D\to\C$ are non-constant and $h_a(0)\in J_{P_0}$ and $h_b(0)\in J_{Q_0}$. By construction of $h_{n,a}$ and the invariance $G_{P_t}(P_t(z))= dG_{P_t}(z)$, we have
\begin{align*}
d^{-n_j}\rho_{n_j,a}^*(\mu_{P,a}) & =d^{-n_j}dd^c\left(G_{P_{\rho_{n_j,a}t}}(a(\rho_{n_j,a}t)\right)=dd^c\left(G_{P_{\rho_{n_j,a}t}}(P_{\rho_{n_j,a}t}^{n_j}(a(\rho_{n_j,a}t))\right)\\
 & =dd^c\left(G_{P_{\rho_{n_j,a}t}}(h_{n_j,a}(t))\right)\longrightarrow_{j\to+\infty} dd^cG_{P_0}\circ h_a=h_a^*(\mu_{P_0})
\end{align*}
and $h_a(0)\in\mathrm{supp}(\mu_{P_0})$. Similarly, $d^{-n_j}\rho_{n_j,b}^*(\mu_{Q,b}) \to h_b^*(\mu_{Q_0})$ as $j\to+\infty$ and $h_b(0)\in\mathrm{supp}(\mu_{Q_0})$. As $\psi^*(\mu_{Q,b})\asymp \mu_{P,a}$, we deduce similarly that, for any weak limit of the sequence of positive measure $d^{-n_j}\rho_{n_j,a}^*(\psi'(0)^*\mu_{Q,b})$
\begin{align*}
\lim d^{-n_j}\rho_{n_j,a}^*(\psi'(0)^*\mu_{Q,b}) \asymp h_a^*(\mu_{P_0}).
\end{align*}
In particular, 
\begin{align*}
\lim \frac{\rho_{n_j,a}}{\rho_{n_j,b}}	 d^{-n_j}\rho_{n_j,b}^*(\psi'(0)^*\mu_{Q,b}) \asymp h_a^*(\mu_{P_0}).
\end{align*}

As $h_a^*(\mu_{P_0})$ and $h_b^*(\mu_{Q_0})$ don't give mass to points, this implies that $(\log(\rho_{n_j,a}/\rho_{n_j,b}))_j$ is a bounded sequence in $\mathbb{R}$. Up to extraction, we thus can assume it converges and up to rescalling, we can assume $\rho_{n_j,a}/\rho_{n_j,b}\to1$ as $j\to\infty$.
We thus have $h_b^*(\mu_{Q_0})\asymp h_a^*(\mu_{P_0})$ as measures on $\D$. 

Since $h_a$ and $h_b$ are non-constant, we can find an open subset $V\subset\D$ such that $h_b:V\to\C$ is injective and $J_{Q_0}\cap h_b(V)\neq \varnothing$. Let $U:=h_b(V)$ and set
\[\sigma:=h_b\circ (h_a|_V)^{-1}:U\longrightarrow \C.\]
As $h_b^*(\mu_{Q_0})\asymp h_a^*(\mu_{P_0})$ as measures on $\D$, we get $\sigma^*(\mu_{P_0})\asymp\mu_{Q_0}$ as measures on $U$.
Since $P_0$ and $Q_0$ satisfy $\mathrm{CE}$, we conclude using Theorem~\ref{tm:DFG}.
\end{proof}

\subsection{Local symmetries of the bifurcation measure of a special curve}
Let $C\subset\mathbb{C}^{d-1}$ be an irreducible algebraic curve that parameterizes a family of polynomials of degree $d$. We say that $C$ is \emph{special} if $C\cap PCF$ is an infinite subset of $C$. The \emph{bifurcation measure of the curve} $C$ is
\[\mu_{\bif,C}:=dd^c(G_\mathrm{bif}|_C).\]

\begin{definition}
A \emph{symmetry} of the measure $\mu_{\bif,C}$ on an open set $U\subset C$ is a non-constant holomorphic map $\phi:U\to C$ such that $\phi^*(\mu_{\bif,C})\asymp  \mu_{\bif,C}$ on $U$.
\end{definition}

We prove here the following which is an analogue to the main result of \cite{Luo}.

\begin{theorem}\label{tm-local-sym-special}
There is a constant $N\geq1$ depending only on $d$ such that for any special curve and any open $U\subset C$ set with $U\cap \mathrm{supp}(\mu_{\bif,C})\neq\varnothing$, there are at most $N$distinct symmetries of $\mu_{\bif,C}$ on $U$.
\end{theorem}

For any $t\in \mathbb{C}^{d-1}$, denote by $\mathrm{Inter}(P_t)$ the set of parameters $s\in \mathbb{C}^{d-1}$ for which there is an irreducible algebraic curve $Z\subset\mathbb{C}^{2}$ which is $(P_t,P_s)$-preperiodic and which projects surjectively onto both coordinates. For the proofs of \eqref{eq_endo_bif} in Theorem~\ref{tm_Functions} and of Theorem~\ref{tm-local-sym-special}, we rely on the result below. 

\begin{proposition}\label{prop:symm}
 For any $d\geq3$, there is a constant $N\geq 1$ depending only on $d$ such that for any $t\in \mathbb{C}^{d-1}$, the set $\mathrm{Inter}(P_t)$
is finite and has cardinality at most $N$.
\end{proposition}
This result is actually a quite straightforward consequence of \cite{book-unlikely} as we 
 briefly explain now:
\begin{proof}
 We say degree $d$ polynomials $P$ and $Q$ are intertwined if there is a (possibly reducible) algebraic curve $Z\subset \mathbb{C}^2$ that is invariant under $(P,Q)$ which projects surjectively onto both coordinates. Being intertwined is an equivalence relation and, given $t$ $\mathrm{Inter}(P_t)$ is the equivalence class of $P_t$(see~\cite[Chapter~4]{book-unlikely}).

It follows immediately from~\cite[Theorem 3.46]{book-unlikely} that there is a constant $N\geq1$ depending only on $d$ such that for any $t\in \mathbb{C}^{d-1}$, the set of monic and centered degree $d$ polynomials $Q$ such that $P_t$ and $Q$ are intertwined is finite and contains at most $N$ elements. As a cubic monic centered polynomial is conjugated to at most $d-1$ polynomials of the form $P_s$, we have
\[\#\mathrm{Inter}(P_t)=\{s\in \mathbb{C}^{d-1}\; ; \ P_t \ \text{and} \ P_s \ \text{are intertwined}\}\leq N(d-1).\]
This concludes the proof.
\end{proof}

\begin{proof}[Proof of Theorem~\ref{tm-local-sym-special}]
Pick an open set $U\subset C$ and assume there is $\sigma:U\to C$ holomorphic and non-constant such that $\sigma^*(\mu_{\bif, C})\asymp \mu_{\bif,C}$ as measures on $U$.
 According to~\cite[Theorem~8.1]{book-unlikely}, since $C$ is a special curve, if we let $\mathsf{P}:=\{1\leq i\leq d-1,\; ; \ G_i|_C\equiv0\}$, then $\# \mathsf{P}<d-1$ and
\begin{enumerate}
\item for any $i\in \mathsf{P}$, there are integers $n_i\geq0$ and $m_i\geq1$ such that $P_t^{n_i}(c_i(t))=P_t^{n_i+m_i}(c_i(t))$ for any $t\in C$,
\item for any $j\notin\mathsf{P}$, there is $\alpha_j>0$ such that $G_j=\alpha_j\cdot G_\bif$ on $C$, the set $\{t\in C\; ; \ c_j(t)$ is preperiodic for for $P_t\}$ coincides with $PCF\cap C$.
\end{enumerate}
Whence we can pick $i$ so that $\mu_{\bif,C}$ is proportional to $\mu_{P,c_i}$. In particular, our assumption that $\sigma^*(\mu_{\bif, C}) \asymp \mu_{\bif,C}$ as measures on $U$ translates as
\[\sigma^*(\mu_{P,c_i})\asymp \mu_{P,c_i}\] measures on $U$. In particular, $\sigma^*$ sends a set of $\mu_{P,c_i}$-full measure in $\sigma(U)$ to a set of $\mu_{P,c_i}$-full measure in $U$.

Moreover, by \cite[Theorem~A]{dTGV} and~\cite[Theorem~4.6]{JiXie}, for $\mu_{P,c_i}$-almost every $t\in C$, 
\begin{enumerate}
\item $P_t$ satisfies $\mathrm{CE}$,
\item for all $i\notin \mathsf{P}$, the pair $(P,c_i)$ satisfies $\mathrm{ParCE}$ at $t$,
\item for all $i\notin \mathsf{P}$, the pair $(P,c_i)$ satisfies $\mathrm{PR}(s)$ at $t$ for some $s>1/2$.
\end{enumerate}
Combined with the above, this implies there is a non-countable set of parameters $t_0\in U$ such that $t_0$ and $\sigma(t_0)$ satisfy $(1)$, $(2)$ and $(3)$. By Proposition~\ref{lm:at-transversely-prerepelling}, this implies $P_{\sigma(t_0)}\in\mathrm{Inter}(P_{t_0})$ for an uncountable set of parameters $t_0\in U$. By Proposition~\ref{prop:symm}, there can be at most $N$ such maps $\sigma$ and the proof is complete.
\end{proof}

\begin{proof}[Proof of Theorem~\ref{tm:special}] Take a special family parameterized by $\C$ and assume that it is the Julia set of a polynomial map $h$ of $\P^1$ of degree $D>1$.  As seen above, the measure $\mu_\bif$ is (proportional) to $dd^c G_\bif$ and is supported by $K:=\{G_\bif=0\}$. Moreover, $G_\bif$ has logarithmic growth at $\infty$ and $G_\bif$ is harmonic on $\mathbb{C}\setminus K$. Thus $\mu_\bif$ is the equilibrium measure of $K$. Let $\mu_h$ denote the Green measure of $h$. As both $\mathrm{supp}(\mu_h)=J_h=\partial K$ and $\mu_h$ is the equilibrium measure of $J_h$, we have $\mu_h=\mu_\bif$. In particular, for any small open ball $U$ such that $U$ does not intersect the critical set $\mathrm{Crit}(h^{N+1})$, $h^j:U\to h^j(U)$ satisfies $(h^j)_*(\mu_\bif)=\mu_\bif$ on $h(U)$ so pulling back, $h$ define a symmetry of $\mu_\bif$. This gives us $N+1$ symmetries of $\mu_\bif$: $h, h^2, \dots, h^{N+1}$ which contradicts Proposition~\ref{prop:symm}.  
\end{proof}

\subsection{Proof of (1) in Theorem~\ref{tm_Functions}}

Le us start with the following.
\begin{lemma}\label{lm:Tbif}
	Let $h$ be a regular polynomial endomorphism of $\mathbb{C}^2$ such that  $\mu_h=\mu_\mathrm{bif}$, then, up to replacing $h$ by $h^2$, we have $h^*T_i= DT_i$ for $i=1,2$. In particular, $h(\{G_i=0\})\subset\{G_i=0\}$.
\end{lemma}
\begin{proof} Assume $\mu_h=\mu_\mathrm{bif}$ so $G_\bif = G_h$. Consider the open sets $U_1:=\{z\in \C^2, \ G_1(z)>G_2(z)>0 \}$ and  $U_2:=\{z\in \C^2, \ G_2(z)>G_1(z)>0 \}$, then $U_1$ and $U_2$ are open, connected. Since $\{z\in \C^2, \ G_2(z)=G_1(z)>0 \}$ is totally invariant by $h$, we can assume, up to taking $h^2$ that $h^{-1}(U_i)=U_i$. In particular, we can apply Lemma~\ref{lm:preserving} to $G_1$ and we deduce that, as in the proof of Proposition~\ref{prop:differentGreen}, that $h^*T_i= DT_i$ for $i=1,2$.
\end{proof}

We say that a parameter $t_0\in \mathbb{C}^2$ is \emph{transversely PCF} if there are integers $n_1\neq m_1,n_2\neq m_2\geq0$ such that $P^{n_i}_{t_0}(c_i(t_0))=P^{m_i}_{t_0}(c_i(t_0))$ and such that the curves $\{P^{n_i}_{t}(c_i(t))=P^{m_i}_{t}(c_i(t))\}$ are smooth and transverse at $t_0$.

\begin{proof}[Proof of \eqref{eq_endo_bif} in Theorem~\ref{tm_Functions}]
We proceed by contradiction assuming there exists $h:\mathbb{C}^2\to\mathbb{C}^2$ such that $\mu_h=\mu_\mathrm{bif}$. First, we show that $h$ sends a transversely PCF parameter to a transversely PCF parameter, as soon as those don't lie in the critical locus $\mathrm{Crit}(h)$ of $h$.
According to Lemma~\ref{lm:Tbif},  up to replace $h$ by $h^2$, we have $h^*T_i=DT_i$, and thus $h(\{G_i=0\})\subset \{G_i=0\}$, for $i=1,2$.

\medskip

Let $C\subset\mathbb{C}^2$ be an irreducible affine curve such that $c_i$ is persistently preperiodic on $C$. Then $h(C)$ is contained in $\{G_i=0\}$. By~\cite[Theorem 2.5]{favredujardin}, this implies that $c_i$ is persistently preperiodic on $h(C)$. Since PCF parameters are isolated intersection points of two such curves $C$ (resp. $C'$) where $c_1$ (resp. $c_2$) is persistently preperiodic and since $h(C\cap C')\subset h(C)\cap h(C')$, $h(p)$ is PCF for any PCF $p\in \mathbb{C}^2$. Moreover, if $C$ and $C'$ are transverse at $p$, then $h(C)$ and $h(C')$ are transverse as soon as $p\notin \mathrm{Crit}(h)$.

\medskip

Let $N\geq1$ be given by Proposition~\ref{prop:symm} and let $p_0\in \C^2\setminus \bigcup_{n\leq N+1}\mathrm{Crit}(h^n)$ be transversely PCF, but not periodic of period $\leq N+1$. Such a point exists since transversely PCF parameters are dense in $\mathrm{supp}(\mu_\mathrm{bif})$, by e.g. \cite[Corollary~0.2]{Dujardin2012}, and thus Zariski dense.

Take an open neighborhood $U\subset \mathbb{C}^2$ of $p_0$ such that for all $1\leq n\leq N+1$, the map $h^n|_U:U\to U_n:=h^n(U)$ is a biholomorphism.  Let $C\subset \mathbb{C}^2$ be the irreducible curve containing $p_0$ consisting of parameters where $c_1$ is preperiodic. Then $c_1$ is also preperiodic on $C_n:=h^n(C)$ from the above and we can set $\phi_n:=h^n|_{U\cap C}=\Lambda_1:=U\cap C\to \Lambda_{2,n}:= U_n\cap C_n$ and $\alpha_n=D^{-n}$. We also let $Q_{t,n}:=P_t$ for $t\in \Lambda_{2,n}$, as well as $a:=c_2|_{\Lambda_1}$ and $b:=c_2|_{\Lambda_{2,n}}$, then we have $\mu_{P,a}:=dd^c(G_2|_{\Lambda_1})$ and $\mu_{Q_n,b}:=dd^c(G_2|_{\Lambda_{2,n}})$ and 
\[\mu_{P,a}=dd^c(G_2|_{\Lambda_1})=\frac{1}{D^n}dd^c(G_2|_{\Lambda_{2,n}}\circ \phi_n)=\alpha_n \cdot \phi_n^*(\mu_{Q_n,b}).\]
We thus can apply Corollary~\ref{cor:at-transversely-prerepelling} at parameters $p_0$ and $p_n:=\phi_n(p_0)=h^n(p_0)$. We thus have $p_n\neq p_m$ for $n\neq m $ and 
\[\{P_{p_1},\ldots,P_{p_{N+1}}\}\subset \mathrm{Inter}(P_{p_0}).\]
This is a contradiction since $\#  \mathrm{Inter}(P_{p_0})\leq N$ by Proposition~\ref{prop:symm}.
\end{proof}

\begin{remark}
\emph{One can easily generalize this statement as follows.} Fix $d\geq4$ and let $h:\mathbb{C}^{d-1}\to\mathbb{C}^{d-1}$ be a polynomial endomorphism which extends holomorphically as $h:\mathbb{P}^{d-1}(\mathbb{C})\to\mathbb{P}^{d-1}(\mathbb{C})$. Then $\mu_h\neq \mu_\bif$.
\emph{The proof follows the lines of that given above. The only additional work resides in proving a generalization of Lemma~\ref{lm:preserving} to $\mathbb{C}^{d-1}$.}
\end{remark}

\subsection{Finitely many preperiodic points are PCF} 
We now prove Corollary~\ref{cor_unif_bif_polynomial}. For the sake of simplicity, we only consider the case where $h$ is defined over $\bar{\Q}$, see the proof of Corollary~\ref{cor_noidea} below to see how we can manage to prove it over $\C$ using a specialization's argument. 

We proceed by contradiction. 

Assume first that $\mathrm{Preper}(h)\cap PCF$ is Zariski dense.
We thus can apply the equidistribution Theorem of Yuan~\cite{yuan}: there is a sequence $p_n\in \mathrm{Preper}(h)\cap PCF$ with
 \[\lim_n \frac{1}{ \# \mathrm{Orb}(p_n)}  \sum_{p \in \mathrm{Orb}(p_n) } \delta_p \to \mu_h \quad \mathrm{and} \quad  \lim_n \frac{1}{ \# \mathrm{Orb}(p_n)}  \sum_{p \in \mathrm{Orb}(p_n) } \delta_p \to \mu_\bif,  \] 
 We thus have $\mu_h=\mu_\bif$. This is impossible by Theorem~\ref{tm_Functions}.

Next, we assume the set $\mathrm{Preper}(h)\cap PCF$ is infinite but not Zariski dense. The Zariski closure of $\mathrm{Preper}(h)\cap PCF$ thus consists of the union of a finite set and of finitely many irreducible curves $C_1,\ldots,C_\ell\subset\C^2$. Each of those curves contains infinitely many points in $\mathrm{Preper}(h)\cap PCF$. By \cite[Corollary B]{DFR}, this implies $C:=C_i$ is preperiodic under iteration of $h$. Up to replacing $h$ by an iterate, we thus can assume $h(C)$ is fixed by $h$ and $Q:=h|_{h(C)}:C\to C$ lifts to a desingularization of $h(C)$ to a polynomial $Q$. Assume $h(C)$ is smooth.

Let $\{\|\cdot\|_v\}_{v\in M_\mathbb{K}}$ be the adelic semi-positive continuous metric on $\mathcal{O}_{\mathbb{P}^2}(1)$ which is $h$ invariant, i.e. such that $h^*\|\cdot\|_v=\|\cdot\|_v^D$ for all $v\in M_\mathbb{K}$. This induces, by restriction, an adelic semi-positive continuous metrization on $L:=\mathcal{O}_{\mathbb{P}^2}(1)|_C$. 
Fix some archimedean place. The curvature current $\mu$ of $(L,\|\cdot\|_v)$ is a positive measure of mass $\deg(C)$ supported on $C$ and if $Q=h|_{h(C)}:h(C)\to h(C)$, then $\mu$ is proportional to $h^*(\mu_Q)|_C$.
As $C$ contains an infinite sequence of points $p_n\in \mathrm{Preper}(h)\cap PCF$, it is also a special curve and one can apply again Yuan's Theorem:
\[\lim_n \frac{1}{ \# \mathrm{Orb}(p_n)}  \sum_{p \in \mathrm{Orb}(p_n) } \delta_p \to \frac{1}{\deg(C)}\mu,\quad  \lim_n \frac{1}{ \# \mathrm{Orb}(p_n)}  \sum_{p \in \mathrm{Orb}(p_n) } \delta_p \to \frac{1}{\deg(C)}\mu_{\bif,C}.  \] 
We thus have $\mu_{\bif,C}=\mu$ as measures on $C$. Take now an open set $U\subset C$ such that $U\cap \mathrm{Crit}(h^n|_C)=\varnothing$ for $1\leq n\leq N+2$, where $N$ is given by Proposition~\ref{prop:symm} and such that $U\cap \mathrm{supp}(\mu)\neq\varnothing$. Then $V:=h(U)\subset h(C)$ satisfies $h(U)\cap J_Q\neq\varnothing$ and $V\cap \mathrm{Crit}(Q^n)=\varnothing$ for $1\leq n\leq N+1$. Up to reducing $V$, and thus $U$, we can assume $Q^n|_V:V\to Q^n(V)$ is a biholomorphism for any $1\leq n\leq N+1$ and the map $h:C\to h(C)$ restricts as a biholomorphism from an open set $U_n\subset C$ to $Q^n(V)$. Denote by $\psi_n$ the local inverse of $h|_C$ and let
\[\sigma_n:=\psi_n\circ Q^n\circ h|_U:U\to C.\]
By construction of $\sigma_n$, we have as measures on $U$:
\begin{align*}
\sigma_n^*(\mu_{\bif,C}) & =(h|_U)^*(Q^n)^*\psi_n^*(\mu_\bif)=(h|_U)^*(Q^n)^*(\alpha_n\cdot  \mu_Q)\\
&=D^n\alpha_n \cdot (h|_U)^*(\mu_Q)= \beta_n \cdot \mu_{\bif,C}.
\end{align*}
If $p_i=\sigma_i(p_0)$, then $\{p_1, \ldots,p_{N+1}\}$ has cardinality $N+1$ whence $\sigma_1,\ldots,\sigma_{N+1}$ are $N+1$ distinct symmetries of $\mu_{\bif,C}$ on $U$. This contradicts Theorem~\ref{tm-local-sym-special}.

\section{The Julia set of a Hénon map is neither the Julia set of an endomorphism nor a bifurcation locus} \label{sec:Henon}
 
 \subsection{Proof of (2) in Theorem~\ref{tm_Functions}}
 The following tells us that a polynomial endomorphism cannot preserve the forward filled Julia set of a Hénon map. We use the notations of sections~\ref{sec_endo} and \ref{sec_henon}.
 \begin{theorem}\label{tm:K+cannotbeinvariant}
 	Let $f$ be a generalized H\'enon map of degree $d>1$ and let $h$ be a regular polynomial endomorphism of $\mathbb{C}^2$ of degree $D>1$. Then $h(\{G^+_f=0\}) \nsubseteq \{G^+_f=0\} $.
 \end{theorem}
 \begin{proof} We proceed by contradiction and assume that $h(\{G^+_f=0\}) = \{G^+_f=0\} $. In particular, $h_*(T_f^+)$ is a positive closed current of degree $D$ supported on $ \{G^+_f=0\}$. By a result of Forn\ae ss and Sibony \cite{MR1332961} (see also \cite{MR3345839}), $T_f^+$ is the unique positive closed current of mass $1$ supported on $\{G^+_f=0\}$ so $h_*(T_f^+)=DT_f^+$. 
 
 \medskip
 
 	Assume first that $h^*(T_f^+)=DT_f^+$ so $G_f^+\circ h-DG^+_f $ is a pluriharmonic function which is a $O(\log^+\|z\|)$ so it is constant. Evaluating at any $p\in  \{G^+_f=0\}$ gives $G_f^+\circ h=DG^+_f$. We consider $h_1 = f \circ h \circ f^{-1}$.
 	It is a polynomial map of $\C^2$, a priori meromorphic on $\P^2$, of topological degree $D^2$. Since $G_f^\pm \circ f^\pm  = d G_f^\pm $ and $G_f^\pm \circ f^\mp  = d^{-1} G_f^\pm $, we deduce:
 	\begin{itemize}
 		\item $G^+_f \circ h_1(z)= dG^+_f \circ h\circ f^{-1}(z)= dDG^+_f \circ f^{-1}(z)= DG^+_f(z)\leq D \log^+\|z\| + O(1)$;
 		\item $G^-_f \circ h_1(z)= d^{-1}G^-_f \circ h\circ f^{-1}(z)\leq  d^{-1}(h\circ f^{-1})^*(\log^+ \| z\| +O(1)) \leq D\log^+ \| z\| +O(1)$
 	\end{itemize}    
 	where we use that $G^\pm_f\leq \log^+\|z\|+O(1)$, that $h\circ f^{-1}$ has degree $dD$, and that the pull back of $\log^+\|z\|$ by a polynomial map of degree $dD$ is $ \leq dD \log^+\|z\| +O(1)$. Finally, as $G_f(z)= \log^+\|z\|+O(1)$, we derive:
 	\[ h_1^*(\log^+\|z\|) \leq D \log^+\|z\|+O(1).\]
 	
 	Assume In particular, $\deg(h_1)\leq D$. From $\deg(h_1)^2 \geq D^2$ (the square of the algebraic degree bounds from above the topological degree), we deduce $\deg(h_1)=D$. This implies that $h_1$ is holomorphic since $\deg(h_1)^2$ is its topological degree. We thus have proved that $h$ and $h_1$ are polynomial endomorphisms of $\C^2$ that extend to endomorphisms of $\P^2$, by the result of Cantat and Xie, \cite[Proposition 8]{cantat2020birational}, they cannot be conjugated by a H\'enon map (only by an affine automorphism). So we reached a contradiction.
 	
 	\medskip
 	
 Assume now $h^*(T_f^+)\neq DT_f^+$.	In particular, we can write  $h^*(T_f^+)=(D-d_1)T_f^++ d_1S$ where $d_1>0$ and $S$ is a positive closed current of mass $1$, not supported on $\{G^+_f=0\}$. Because, $T_f^+$ is laminar \cite{BedfordLyubichSmillie2}, $S$ also is.  Note that $h_*(S)= DT^+_f$.  
 	For $a >0$, let $G_a:= \max (G_f,a)$ and $T_a:= dd^c G_a$.  Let us recall one the main results of \cite[Theorem 7.2]{HOV}: each current $T_a$ is laminated by leaves which are dense  in $\mathrm{supp}(T_a)$. 
 	
 	As $h_*(S)\wedge T_a= DT^+_f \wedge T_a=0$, we deduce $S\wedge h^*(T_a)= 0$.  Then, consider the family of  currents $h^*(T_a)$, for all $a>0$. Since $h$ is holomorphic, each  $h^*T_a$ is again laminated by leaves which are dense  in $\mathrm{supp}(T_a)$. 
 	As $h$ is surjective, for each $z \in \C^2\backslash \{G^+_f=0\}$ at which passes a leaf $L$ of the lamination of $S$, we can find $a$ such that $z \in \mathrm{supp}(h^*(T_a))$. In particular, one deduces by laminarity that the leaf $L$ is also a leaf of the lamination of $h^*(T_a)$ so $\supp(h^*(T_a)) \subset \supp(S)$. This is a contradiction since then we would have $\supp(T_a)=h(\supp(h^*(T_a))) \subset h(\supp(S)) = \supp(T^+_f)$.  
 \end{proof}
 
 We deduce the following proposition which is exactly \eqref{eq_endo_henon} in Theorem~\ref{tm_Functions}.  
 \begin{proposition}\label{prop:differentGreen}
 	Let $f$ be a generalized H\'enon map of degree $d>1$ and let $h$ be a regular polynomial endomorphism of $\mathbb{C}^2$ of degree $D>1$. Then $G_f\neq G_h$.
 \end{proposition}
 \begin{proof}
 	Again, we proceed by contradiction and assume $G_f=G_h$. In particular,  $G_f \circ h = D G_f$. Consider $U^+=\{z \in \C^2, \ G_f^+(z)> G_f^-(z)>0\}$ and $U^-=\{z \in \C^2, \ G_f^-(z)> G_f^+(z)>0\}$, then $U^\pm$ are disjoint open sets, one can check they are connected. In particular, since $h$ is surjective and invariance of $\{G_h= 0\}$ and $\supp (dd^cG_h)= \{z, G^+_f(z)=G^-_f(z)\}$, we have $h( U^+) = U^+$ and $h (U^-) = U^-$ or $h (U^+) = U^-$ and $h (U^-) = U^+$. As iterating does not change the problem, we assume $h^{-1} (U^+) = U^+$ and $h^{-1}( U^-) = U^-$.
 	
 	So by Lemma~\ref{lm:preserving}, applied to $G_f^+$ on $U^+$ (where $G_f^+ \circ h = G_f \circ h= D G_f=DG^+_f$), we deduce that $G_f^+\circ h =D G_f^+$ on $\{G_f>0 \}$ and by $h^{-1}\{G_f>0 \}=\{G_f>0 \} $, this stands in $\C^2$ and so $G_f^+\circ h =D G_f^+$. In particular, $h^*(T_f^+) =D T_f^+$, this contradicts Theorem~\ref{tm:K+cannotbeinvariant}.
 \end{proof}

 \subsection{Common (pre)periodic points for H\'enon maps and endomorphisms}\label{sec_arith}
 This section is devoted to the proof of 
 \begin{corollary}\label{cor_noidea}
 	Let $h:\C^2 \to \C^2$  a polynomial endomorphism of $\P^2(\C)$ of degree $D>1$ and let $f:\C^2\to \C^2$ be a generalized H\'enon map of degree $d>1$. 
 	\begin{enumerate}
 		\item If $h$ and $f$ are both defined over a number field, the set $\mathrm{Preper}(h)\cap \mathrm{Per}(f)$ is not Zariski dense,
 		\item If $|\mathrm{Jac}(f)| \neq 1$, the set $ \mathrm{Preper}(h)\cap \mathrm{Per}(f)$ is finite.
 		\end{enumerate}
 \end{corollary}
The second point is exactly point (1) in Corollary~\ref{cor_unif_henon_polynomial}.
\begin{proof}

 Let $f:\C^2\to \C^2$ be a generalized H\'enon map and  $h:\C^2 \to \C^2$  a polynomial endomorphism of $\P^2(\C)$ of degree $D>1$ both defined over a number field $\mathbb{K}$. Assume that they have a Zariski dense set $E$ of $\mathrm{Per}(f)\cap \mathrm{Preper}(h)$. This set is defined over $\bar{\mathbb{Q}}$, and we can extract a generic sequence $p_n \in E$. Then, at the complex place, the equidistribution of points of small height for H\'enon maps (\cite{Lee-Henon}) and endomorphisms of $\P^2$ (\cite{yuan}) implies
 \[\lim_n \frac{1}{ \# \mathrm{Orb}(p_n)}  \sum_{p \in \mathrm{Orb}(p_n) } \delta_p \to \mu_f \quad \mathrm{and} \quad  \lim_n \frac{1}{ \# \mathrm{Orb}(p_n)}  \sum_{p \in \mathrm{Orb}(p_n) } \delta_p \to \mu_h,  \] 
 where $ \mathrm{Orb}(p_n)$ is the Galois orbit of $p_n$ and the convergences are in the sense of measures. In particular, this contradicts \eqref{eq_endo_henon} in Theorem~\ref{tm_Functions}.

 \bigskip
 
 We are now in position to prove the second point using a specialization argument of Dujardin and Favre \cite{DujardinFavre_MM}. 
 
 We assume $|\mathrm{Jac}(f)| \neq 1$. When $f$ is defined over a number field, by the above, the set $\mathrm{Per}(f)\cap\mathrm{Preper}(h)$ is not Zariski dense. Assume now it is infinite. In this case, there is a Zariski closed subset $Z$ such that $f$ admits infinitely many periodic points on $Z$, which contradicts the Manin-Mumford problem for polynomial automorphisms of Dujardin and Favre \cite[Theorem~A]{DujardinFavre_MM}.
 
 In the general case, we rely on a specialization argument. Let $E:=\mathrm{Preper}(h)\cap\mathrm{Per}(f)$. Pick a finitely generated $\bar{\mathbb{Q}}$-algebra $R \subset\mathbb{C}$ such that $f$, $f^{-1}$ and $h$ are all defined over $K:=\mathrm{Frac}(R)$, and let $S:=\mathrm{Spec}(R)$.
 
 For any $t\in S$, let $f_t$, $f_t^{-1}$ and $h_t$ be the specializations at $t$ of $f$, $f^{-1}$ and $h$ respectively. As in \cite[Lemma~6.6]{DujardinFavre_MM}, up to replacing $S$ by a dense Zariski open subset $S'$, we can assume $f_t$ is a generalized H\'enon map and $h_t$ is a regular polynomial endomorphism of $\mathbb{C}^2$ for all $t\in S$. For $t\in S$, denote by $E_t$ the specialization of $E$ at $t$. 
 Proceeding as in \cite[Proof of Theorem~D]{DujardinFavre_MM}, we have
 \begin{lemma}\label{lm:specialization}
 	If the set $E$ is infinite, the set $E_{t}$ is also infinite for any $t\in S(\bar{\Q})$. 
 \end{lemma}
 Assume the lemma holds.
 Pick $t_0\in S(\bar{\mathbb{Q}})$ such that $|\mathrm{Jac}(f_{t_0})|\neq1$. Such a parameter exists, since otherwise $\mathrm{Jac}(f)$ would be an algebraic number with $|\mathrm{Jac}(f)|=1$. By the first step of the proof, the set $E_{t_0}$ is finite, whence $E$ is finite.
 \end{proof}
 \medskip
 
 The proof of Lemma~\ref{lm:specialization} is directly inspired by \cite{DujardinFavre_MM}.
 \begin{proof}[Proof of Lemma~\ref{lm:specialization}]
 	Fix $t$ and assume by contradiction that $E_t$ is finite. Then $E_t$ is included in the set of fixed points of $f^{n_0}$ for some $n_0\geq1$. For each $n \geq n_0$, let $Y_n$ be the subscheme of $\mathbb{A}^2_S$ whose underlying space is
 	$E\cap \{(x, y) : f^n(x, y) = (x, y)\},$
 	endowed with the scheme structure induced by the quotient sheaf
 	$\mathcal{O}_{\mathbb{A}^2_S}/(f_1^n-x,f_2^n-y)$, where $f^n=(f_1^n,f_2^n)$.
 	For any $z\in E_t$ , the ordinary multiplicity $e(z, Y_{n,t})$ of $z$ as a point in $Y_{n,t}$ is equal to the multiplicity $\mu(z, f_t^n)$ as a fixed point for $f_t^n$. By the Shub-Sullivan Theorem \cite{SS74}, the sequence $e(z,Y_{n,t})$ is bounded. Whence $\sum_{z\in E_t}e(z,Y_{n,t})$ is bounded.
 	Arguing as in the proof of \cite[Lemma 5.3]{DujardinFavre_MM}, the map $Y_n\to S$ is proper and finite, and by Nakayama's lemma we get
 	\[\#[E\cap\{f^n=\mathrm{id}\}]\leq \sum_{z\in E} e(z,Y_n)\leq   \sum_{z\in E_t} e(z,Y_{n,t})\leq C<\infty,\]
 	where $C$ is a constant independent of $n$. Since $E \cap \{f^{n!} = \mathrm{id}\}$ contains all the periodic points in $E$ of period at most $n$, the cardinality of $E \cap \{f^{n!} = \mathrm{id}\}$ tends to infinity as $n\to\infty$. This is  a contradiction, proving that $E_t$ is infinite.
 \end{proof}

\subsection{ Proof of  (3) in Theorem~\ref{tm_Functions}}
We show \eqref{eq_henon_bif} in Theorem~\ref{tm_Functions} by contradiction and assume  $G_f=\max(G^+_f, G^-_f)= G_\bif$. Recall that $U^+=\{(x,y) \in \C^2, \ G_f^+(x,y)> G_f^-(x,y)>0\}$ and $U^-=\{(x,y) \in \C^2, \ G_f^-(x,y)> G_f^+(x,y)>0\}$ are disjoint connected open sets. In particular, we can assume that $U^+=\{ (x,y),  G_1(z)> G_2(z)>0\}$ so that $G^+_f=3G_1$ on $U^+$ (and similarly, $G^-_f=G_2$ on $U^-$). By connectivity of $\{\max(G^+_f, G^-_f)>0\}$, and analyticity  we deduce that $G^+_f=3G_1$ on $\C^2$, so $3T_1=T^+_f$. 
 	
 	Now, we can remark that, for any $(n,m)\in \N \times \N^* $, the set $\{G_1=0\}$ contains algebraic curves of the form
 	\[C_{n,m}:=\{  (x,y), \ P_{x,y}^{n+m}(c_1(x,y))= P_{x,y}^{n}(c_1(x,y))\}.\]
 	In particular, the set $K_f^+=\{G_f^+=0\}$ contains the curve $C_{n,m}$, whence the current $[C_{n,m}]$ is supported on $K_f^+$. By extremality of $T_f^+$ (\cite{FS2}), this implies $[C_{n,m}]=\deg(C_{n,m})\cdot T_f^+$. This is a contradiction, since $T_f^+$ has continuous potential on $\mathbb{C}^2$.

\medskip
 
As previously, one can deduce the following.
 \begin{corollary}
	Let $f:\mathbb{C}^2\to\mathbb{C}^2$ be a generalized H\'enon map of degree $d>1$.
	\begin{enumerate}
		\item  If $f$ is defined over a number field, then $\mathrm{Per}(f)\cap PCF$ is not Zariski dense, 
		\item If $f$ $|\mathrm{Jac}(f)| \neq 1$, then $\mathrm{Per}(f)\cap PCF$ is finite.
	\end{enumerate} 
\end{corollary}
which implies point (2) in Corollary~\ref{cor_unif_henon_polynomial}. 

\begin{proof}
The proof is the same than the proof of corollary~\ref{cor_noidea} where we instead rely on the equidistribution of generic sequences of PCF parameters from~\cite{favregauthier}. 	
\end{proof}

\bibliographystyle{alpha}
\bibliography{biblio}

\end{document}